\RequirePackage{fix-cm}

\documentclass[smallextended]{svjour3}
\smartqed

\usepackage{graphicx}
\usepackage{amscd,amsfonts,amssymb,amsmath,latexsym,array,hhline,xcolor}
\usepackage{mathtools}
\usepackage{booktabs}
\usepackage{enumitem}
\usepackage{times}

\newcommand{\bbz}{\mathbb{Z}}
\newcommand{\bbr}{\mathbb{R}}
\newcommand{\bbe}{\mathbb{E}}
\newcommand{\bbp}{\mathbb{P}}
\newcommand{\cL}{\mathcal{L}}
\newcommand{\cF}{\mathcal{F}}
\newcommand{\cE}{\mathcal{E}}
\newcommand{\eps}{\varepsilon}
\newcommand{\one}{\mathbf{1}}

\begin{document}

\title{Exact asymptotics of the ruin probability in the Sparre Andersen model for non-life insurance with investments in a L\'evy process}

\titlerunning{Exact asymptotics of the ruin probability in the Sparre Andersen model}

\author{Platon Promyslov}

\institute{\at
Faculty of Mechanics and Mathematics, Lomonosov Moscow State University, Moscow, Russia\\
Faculty of Computer Science, HSE University, Moscow, Russia\\
\email{platon.promyslov@gmail.com}
}

\date{}

\maketitle

\begin{abstract}
We establish the exact power-law asymptotics of the ruin probability, as a function of the initial capital, in the Sparre Andersen model for non-life insurance with investments in an arbitrary L\'evy process. The main advance over previous work, where only the two-sided order was obtained, is a proof of the existence of an exact limiting equality with a positive finite constant. The method combines a reduction to discrete time, the one-dimensional Kesten--Goldie theorem for the stationary measure of the associated affine stochastic recursion, and Goldie's result on the asymptotics of the supremum of a perpetuity. The limiting constant is bounded below by the integral Goldie constant of the stationary measure. The exponent coincides with the positive Cram\'er root of the Laplace exponent of the L\'evy process given by the logarithm of the risky-asset price; the distribution of the inter-jump times of the business process affects only the constant, not the exponent. In the final section the constant is compared with an explicit formula in terms of double confluent Heun functions, obtained for the Cram\'er--Lundberg submodel with proportional investment in a geometric Brownian motion; the agreement of the two approaches is illustrated numerically.
\end{abstract}

\keywords{Ruin probability \and Sparre Andersen model \and Non-life insurance \and Implicit renewal theory \and Kesten--Goldie theorem \and Perpetuity equation}

\subclass{60G50; 60G51; 60K05; 91G05}

\section{Introduction}

The study of the ruin probability of an insurance company in the presence of risky investments goes back to the work of Paulsen \cite{Paulsen1993}, where the dramatic change in the nature of the asymptotics of the ruin probability $\Psi(u)$ was first observed---from exponential, in models without investments, to power-law, in models with investments in a risky asset. In parallel, Kalashnikov and Norberg \cite{KalashnikovNorberg2002} obtained power-law asymptotics in the Cram\'er--Lundberg model with risky investments by applying the Kesten--Goldie theory to the associated affine stochastic equation. The theory has developed along two lines: the analytic approach, via integro-differential equations \cite{Paulsen2002,Grandits2004,GaierGrandits2004,AK2026,P2601,P2604,PRYu2026}, and the probabilistic approach, via distributional equations and implicit renewal theory \cite{BDM2016,EKS2022,Goldie1991,KLP2026,KP2023}. A modern exposition of the latter is contained in the monograph \cite{BDM2016}, and the general theory of ruin in \cite{AsmussenAlbrecher2010}.

Two-sided estimates for the order of the ruin probability in the Cram\'er--Lundberg model with investments in a geometric Brownian motion were obtained by Pergamenshchikov and Zeitouni \cite{PergamenshchikovZeitouni2006}. Eberlein, Kabanov and Schmidt \cite{EKS2022} considered the non-life insurance case in the Sparre Andersen model with investments in an arbitrary L\'evy process and established the two-sided estimate for the ruin probability
\begin{equation}\label{eq:liminflimsup}
0<\liminf_{u\to\infty}u^\beta\Psi(u)\le\limsup_{u\to\infty}u^\beta\Psi(u)<\infty,
\end{equation}
where $\beta>0$ is the unique positive root of the cumulant $H(q):=\ln\bbe[e^{-qV_{T_1}}]$, and $V$ is the logarithm of the risky-asset price. In \cite{KLP2026} the conditions of these estimates were substantially weakened, and analogous two-sided estimates were obtained in the mixed model using Proposition 2.5.4 of \cite{BDM2016}. The existence of an exact limit between the two extremes in \eqref{eq:liminflimsup} remained an open question in the Sparre Andersen model for the non-life case with an arbitrary L\'evy process.

In parallel, in \cite{P2601}, by reducing the problem to an integro-differential equation, an exact limit $\Psi(u)\sim Cu^{-\beta}$ with an explicit computable constant $C$ was obtained for a particular case---the Cram\'er--Lundberg model for the life insurance case with proportional investment in a geometric Brownian motion and a constant fraction $\kappa\in(0,1]$ invested in the risky asset. In \cite{AK2026}, the existence of a classical $C^2$-smooth solution to the integro-differential equation for an analogous model, but for the non-life insurance case, was established under the moment condition $\bbe[\xi^{\gamma-1}]<\infty$. This condition was relaxed in \cite{P2604}, where $C^2$-smoothness was proved under the minimal moment condition $\bbe[\xi^\varepsilon]<\infty$ for some $\varepsilon>0$, which improves the classical result of Grandits \cite{Grandits2004}, that yielded $W^{2,1}_{\mathrm{loc}}$-regularity under the same moment condition. A subsequent analysis via special functions in \cite{PRYu2026} gave a representation in terms of double confluent Heun functions. The extension to the general model---the Sparre Andersen model with an arbitrary L\'evy process---requires an essentially different apparatus and is the content of the present work.

The main result of the work is the following: in the Sparre Andersen model for non-life insurance with investments in a general L\'evy process, the two-sided estimate \eqref{eq:liminflimsup} turns into the exact equality
\begin{equation}\label{eq:exact-intro}
\lim_{u\to\infty}u^\beta\Psi(u)=C^*\in(0,\infty),
\end{equation}
where $C^*\ge C_+^Y$, and $C_+^Y$ is the integral Goldie constant of the stationary measure of the associated affine recursion, having the explicit representation \eqref{eq:C-plus-Y}.

The method rests on three facts. First, the reduction to a discrete-time model from \cite{EKS2022}: in the non-life case, ruin occurs only at the jump times $T_n$ of the business process, which reduces the problem to the analysis of the tail of the supremum of the discrete chain $\sup_n Y_n\ge u$, where $Y_n:=Y_{T_n}$. Second, the Kesten--Goldie theorem in its one-dimensional form (\cite[Theorem 2.4.4]{BDM2016}, going back to the works of Kesten \cite{Kesten1973} and Goldie \cite{Goldie1991}) gives the power-law asymptotics of the stationary measure $\bbp(Y_\infty>u)\sim C_+^Y u^{-\beta}$. Third, Goldie's result \cite{Goldie1991} on the fixed-point equation $M\stackrel{d}{=}(AM+B)_+$ provides the passage from the asymptotics of the stationary measure to the asymptotics of the supremum, $\bbp(\sup_n Y_n>u)\sim C^*u^{-\beta}$. Finally, the passage from $\bbp(\sup_n Y_n>u)$ to $\Psi(u)=\bbp(\sup_n Y_n\ge u)$ is carried out by means of the two-sided bound $\bbp(\sup_n Y_n>u)\le\Psi(u)\le\bbp(\sup_n Y_n>u-\eps)$ and a subsequent limiting argument. We note that this two-sided bound and the limiting argument require no additional conditions.

Conceptually, the exponent $\beta$ admits a twofold interpretation: as the Cram\'er root of the cumulant $H$ of the random walk $V_{T_n}$ (the formulation of \cite{EKS2022,KLP2026,KP2023}) or, equivalently, as the Cram\'er root of the cumulant $\psi_V$ of the L\'evy process $V$ itself, since
\[
H(\beta)=0\iff\bbe[e^{T_1\psi_V(\beta)}]=1\iff\psi_V(\beta)=0,
\]
by the strict monotonicity of the function $s\mapsto\bbe[e^{sT_1}]$ (a consequence of $T_1>0$ a.s.). This duality gives a transparent interpretation: the exponent is determined by the investment structure $V$ itself, while the distribution of the inter-jump times affects only the constant $C^*$ in \eqref{eq:exact-intro}.

The transfer of the result to the annuity and mixed models is a natural direction for further research and will be the subject of separate papers of the series. In the annuity model the trajectories of $Y$ have a positive continuous drift between jumps, so that $\sup_t Y_t\ne\sup_n Y_{T_n}$ and the discrete reduction fails. In the mixed model $F_\xi$ charges both half-axes, and the left tail $\bbp(Y_\infty<-u)$ requires a separate analysis in the spirit of Goldie's theorem applied to the transform $-Y_\infty$.

The paper is organized as follows. Section \ref{sec:model} describes the model, states the assumptions and the non-lattice condition, gives the reduction to a discrete-time model in the spirit of \cite{EKS2022}, and establishes the convergence of the associated perpetuity. Section \ref{sec:main} states and proves the main result \eqref{eq:exact-intro}. Section \ref{sec:CLgBm} establishes the connection with the exact analytic solution \cite{PRYu2026} for the Cram\'er--Lundberg submodel with proportional investment and gives a numerical verification. Section \ref{sec:discussion} contains a discussion and directions for further research.

\section{The Model and Assumptions}\label{sec:model}

All random elements are defined on a stochastic basis $(\Omega,\cF,(\cF_t)_{t\ge0},\bbp)$ satisfying the usual conditions.

Let $R=(R_t)_{t\ge0}$ be a L\'evy process with triplet $(a,\sigma^2,\Pi)$, $\Pi((-\infty,-1])=0$. The stochastic exponential $\cE(R)=e^V$ is interpreted as the price of the risky asset. We introduce the log-price
\[
V_t=at-\tfrac12\sigma^2 t+\sigma W_t+h*(\mu-\nu)_t+\bigl(\ln(1+x)-h\bigr)*\mu_t,
\]
where $h(x)=x\one_{\{|x|\le1\}}$, $\mu$ is the jump measure of $R$, and $\nu$ is its compensator. Then $V$ is again a L\'evy process. By $\psi_V(q):=\ln\bbe[e^{-qV_1}]$ we denote its Laplace exponent.

The business process $P$ is independent of $R$ and has the form $P_t=ct+\sum_{i=1}^{N_t}\xi_i$, where $N=(N_t)_{t\ge0}$ is a renewal counting process with jump times $0<T_1<T_2<\dots$; the inter-jump times $U_i:=T_i-T_{i-1}$, $i\ge1$, form a sequence of independent identically distributed random variables with law $F$, where $F(\{0\})=0$, and the jumps $\xi_i:=\Delta P_{T_i}$ are independent identically distributed random variables with law $F_\xi$, $F_\xi(\{0\})=0$. The sequences $(U_i)$ and $(\xi_i)$ are mutually independent. We set $T_0:=0$.

In the present work we consider the non-life insurance case: $c>0$ and $\xi_i<0$ almost surely. The process is observed starting from the moment $T_0:=0$, which corresponds to the standard setting of \cite{EKS2022}.

The capital process $X^{u}$ satisfies the linear stochastic equation
\[
X_t^{u}=u+\int_0^t X_{s-}^{u}\,dR_s+P_t,\qquad X_0^{u}=u>0.
\]
We define the ruin time $\tau^{u}:=\inf\{t\ge0:X_t^{u}\le0\}$ and the ruin probability $\Psi(u):=\bbp(\tau^{u}<\infty)$.

By the stochastic Cauchy formula \cite[Corollary 3.2]{KP2023},
\[
X_t^{u}=e^{V_t}(u-Y_t),\qquad Y_t:=-\int_{(0,t]}e^{-V_{s-}}\,dP_s.
\]
The strict positivity of $e^{V_t}$ yields $\tau^{u}=\inf\{t\ge0:Y_t\ge u\}$.

\begin{lemma}\label{lem:Y-structure}
In the non-life insurance case, the trajectories of $Y$ decrease continuously on each interval $(T_{n-1},T_n)$ and have positive jumps at the moments $T_n$. Consequently,
\begin{equation}\label{eq:Psi-sup}
\Psi(u)=\bbp\Bigl(\sup_{n\ge0}Y_{T_n}\ge u\Bigr).
\end{equation}
\end{lemma}

\begin{proof}
Between jumps $dP_s=c\,ds$, hence $dY_s=-e^{-V_{s-}}c\,ds<0$. At the moment $T_n$ the jump equals $\Delta Y_{T_n}=-e^{-V_{T_n-}}\xi_n>0$ because $\xi_n<0$. Since $Y$ decreases between jumps, the supremum over $[T_n,T_{n+1})$ is attained at $T_n$. Hence $\sup_{t\ge0}Y_t=\sup_{n\ge0}Y_{T_n}$, and identity \eqref{eq:Psi-sup} follows from $\tau^{u}=\inf\{t:Y_t\ge u\}$.

\end{proof}

Lemma \ref{lem:Y-structure} is a particular case of the reduction to a discrete-time model stated explicitly in \cite{EKS2022}: in the non-life insurance model it suffices to track the reserve process along the sequence of jump times $T_n$. Set $Y_n:=Y_{T_n}$ and $M^*:=\sup_{n\ge0}Y_n$. By \eqref{eq:Psi-sup}, $\Psi(u)=\bbp(M^*\ge u)$.

Put
\[
M_m:=e^{-(V_{T_m}-V_{T_{m-1}})},\qquad Q_m:=-\int_{T_{m-1}}^{T_m}e^{-(V_{s-}-V_{T_{m-1}})}\,dP_s,\qquad m\ge1.
\]
By the independence of the increments of the L\'evy process $V$, the independence of $P$ and $V$, and the independence and identical distribution of the inter-jump times $U_i$, the pairs $(M_m,Q_m)_{m\ge1}$ form a sequence of independent identically distributed random pairs. A direct computation gives the explicit representation of the chain $(Y_n)$:
\begin{equation}\label{eq:Y-sum}
Y_n=\sum_{k=1}^n\Bigl(\prod_{i=1}^{k-1}M_i\Bigr)Q_k,\qquad Y_0=0.
\end{equation}
We emphasize that \eqref{eq:Y-sum} is not the affine recursion $Y_n=M_nY_{n-1}+Q_n$. The latter holds for the backward (time-reversed) sequence
\begin{equation}\label{eq:tildeY}
\widetilde Y_0:=0,\qquad \widetilde Y_n:=M_n\widetilde Y_{n-1}+Q_n,\quad n\ge1,
\end{equation}
for which, for each fixed $n$, one has $\widetilde Y_n\stackrel{d}{=}Y_n$. Indeed, $\widetilde Y_n=\sum_{k=1}^n(\prod_{i=k+1}^nM_i)Q_k$ is obtained from \eqref{eq:Y-sum} by reversing the order of the pairs $(M_i,Q_i)$, and since these pairs are independent and identically distributed, reversing the order does not change the joint law. Associated with the recursion \eqref{eq:tildeY} is the distributional fixed-point equation
\begin{equation}\label{eq:fixedpoint}
Y_\infty\stackrel{d}{=}M_1\,Y_\infty+Q_1\qquad (Y_\infty\text{ independent of }(M_1,Q_1)),
\end{equation}
the convergence to whose unique solution is established below, after the assumptions are stated.

We retain the standard assumptions of \cite{KP2023}:
\begin{itemize}[leftmargin=*,itemsep=2pt]
\item[\textbf{(A1)}] The process $R$ is non-degenerate: $\sigma^2>0$ or $\Pi\not\equiv0$.
\item[\textbf{(A2)}] The cumulant $H(q):=\ln\bbe[e^{-qV_{T_1}}]$ has a unique root $\beta>0$ lying strictly inside $\mathrm{int}(\mathrm{dom}\,H)$.
\item[\textbf{(A3)}] $\bbe[|\xi_1|^\beta]<\infty$ and $\bbe[e^{\eps T_1}]<\infty$ for some $\eps>0$.
\end{itemize}

For the application of the Kesten--Goldie theorem (Theorem 4.1 of \cite{Goldie1991}) an additional non-lattice condition is required:
\begin{itemize}[leftmargin=*,itemsep=2pt]
\item[\textbf{(N)}] The law of the random variable $V_{T_1}$ is non-arithmetic: there exist no $\delta>0$ and $a\in\bbr$ such that $\bbp(V_{T_1}\in a+\delta\bbz)=1$; equivalently, $|\bbe\,e^{itV_{T_1}}|<1$ for all $t\ne0$.
\end{itemize}

\begin{lemma}%\label{lem:N-diff}
If $\sigma^2>0$, then \textup{(N)} holds.
\end{lemma}

\begin{proof}
If $\sigma^2>0$, then, conditionally on $T_1=t$, the variable $V_{T_1}$ contains the Gaussian component $\sigma W_t-(\sigma^2/2)t$, which has an absolutely continuous density. The unconditional law of $V_{T_1}$, being a mixture, has an absolutely continuous component, which rules out concentration on any coset $a+\delta\bbz$ (a Lebesgue-null set); hence \textup{(N)} holds.
\end{proof}

\begin{lemma}%\label{lem:N-jumps}
If the L\'evy measure $\Pi_V$ has an absolutely continuous component charging an open interval in $\bbr$, then \textup{(N)} holds.
\end{lemma}

\begin{proof}
Let $I\subset\bbr$ be an open interval on which $\Pi_V$ has an absolutely continuous component. The restriction $\Pi_V|_I$ is finite (since $\Pi_V$ is integrable outside any neighborhood of zero by the definition of a L\'evy measure; if $0\in I$, one may pass to a subinterval not containing zero). Consider the decomposition of the L\'evy process $V$ into the sum of two independent processes: a pure-jump process $J^I_t$ with L\'evy measure $\Pi_V|_I$ (a compound Poisson process with finite intensity, the law of whose individual jump has an absolutely continuous component) and the remainder $V^{c}_t:=V_t-J^I_t$. The law of $J^I_t$ for $t>0$ is a shifted mixture over the number of jumps,
\[
\bbp(J^I_t\in\cdot)=\sum_{n\ge0}\frac{(\Pi_V(I)\cdot t)^n}{n!}e^{-\Pi_V(I)\cdot t}\,\nu_I^{*n},
\]
where $\nu_I:=\Pi_V|_I/\Pi_V(I)$ is a probability measure having an absolutely continuous component. Then already $\nu_I^{*1}=\nu_I$ has an absolutely continuous component. Consequently the mixture $\bbp(J^I_t\in\cdot)$ for $t>0$ has an absolutely continuous component. The independent convolution with the law of $V^c_t$ preserves the absolutely continuous component. Hence $V_t$ has an absolutely continuous component for every $t>0$. Averaging over the random moment $T_1$ with a law not concentrated at zero (that is, $F(\{0\})=0$) preserves the presence of an absolutely continuous component in $V_{T_1}$, which implies \textup{(N)}.
\end{proof}

\begin{remark}
Note that the law $F$ of the inter-jump times does not automatically yield the non-arithmeticity of $V_{T_1}$. For example, if $V$ is a standard Poisson process, then $V_t\in\bbz_+$ almost surely for every $t$, and hence $V_{T_1}\in\bbz_+$ for any law of $T_1$, which is lattice with span $1$. Sufficient conditions for (N) must be imposed on the structure of $V$ itself. This example illustrates only the lattice mechanism and does not satisfy (A2): for $V_{T_1}\ge0$ the cumulant $H$ has no positive root.
\end{remark}

\begin{remark}
If condition \textup{(N)} fails, that is, $V_{T_1}$ is lattice with span $\delta>0$, the Kesten--Goldie theorem in its direct form is inapplicable. In this case classical renewal theory (Feller \cite[Chapter XI]{Feller1971}, \cite[Section 2.4]{BDM2016}) yields an asymptotics of the form $u^\beta\Psi(u)=L(u)+o(1)$, where $L$ is a function periodic in $\log u$ with period $\delta$; then $u^\beta\Psi(u)$ oscillates and has no limit, so a relation with $\sim$ is not correct. The lattice case is not considered in the present work.
\end{remark}

Let us verify the conditions for the convergence of the perpetuity. From the convexity of $H$, the equality $H(0)=0$, and the existence of a root $\beta>0$ (assumption (A2)) it follows that $\bbe[\log M_1]=-\bbe[V_{T_1}]=H'(0)<0$. From (A3) we obtain $\bbe[\log_+|Q_1|]<\infty$, since $\bbe[|Q_1|^\beta]<\infty$ by \cite[Lemma 2.1, Corollary 2.2]{EKS2022}.

\begin{lemma}\label{lem:Y-convergence}
Under assumptions \textup{(A1)--(A3)} the forward sum $(Y_n)$ converges almost surely to a finite random variable $Y_\infty^{\mathrm{fwd}}$, while the backward recursion $\{\widetilde Y_n\}$ from \eqref{eq:tildeY} converges in distribution to a random variable $Y_\infty$ satisfying equation \eqref{eq:fixedpoint}. The law $\cL(Y_\infty)$ is the unique stationary measure of \eqref{eq:fixedpoint}, and $Y_\infty^{\mathrm{fwd}}\stackrel{d}{=}Y_\infty$. In what follows $Y_\infty$ denotes a random variable with law $\cL(Y_\infty)$.
\end{lemma}

\begin{proof}
The almost sure convergence $Y_n\to Y_\infty^{\mathrm{fwd}}$ is established in \cite[Lemma 4.1]{EKS2022}. It can also be seen directly: the terms of the series $\sum_k(\prod_{i<k}M_i)Q_k$ decay geometrically, since $\bbe[\log M_1]<0$. The convergence in distribution of the recursion $\widetilde Y_n$ to a solution of \eqref{eq:fixedpoint} and the uniqueness of the stationary measure, under the conditions $\bbe[\log|M_1|]<0$, $\bbe[\log_+|Q_1|]<\infty$ and non-degeneracy, are provided by the Goldie--Maller theorem \cite[Theorem 2.1.1, Corollary 2.1.2]{BDM2016} and \cite{Vervaat1979}. The recursion $\widetilde Y_n$ itself does not converge almost surely, since the difference $\widetilde Y_n-M_n\widetilde Y_{n-1}=Q_n$ does not tend to zero; it is precisely the forward sum $Y_n$ that converges almost surely. Finally, the equality of laws $Y_\infty^{\mathrm{fwd}}\stackrel{d}{=}Y_\infty$ follows from the equality $Y_n\stackrel{d}{=}\widetilde Y_n$ established above for each $n$ and the passage to the limit.
\end{proof}

All asymptotic statements in the work depend only on the law $\cL(Y_\infty)$, so in what follows we work with $Y_\infty$ as a typical representative of this law.

\section{Main Results}\label{sec:main}

The main result of the present work is the following theorem.

\begin{theorem}\label{thm:main}
Under assumptions \textup{(A1)}, \textup{(A2)}, \textup{(A3)} and \textup{(N)},
\begin{equation}\label{eq:main}
\lim_{u\to\infty}u^\beta\Psi(u)=C^*\in(0,\infty),
\end{equation}
where the constant $C^*$ is given by the limit $C^*=\lim_{u\to\infty}u^\beta\bbp(M^*>u)$ and satisfies $C^*\ge C_+^Y$, where $C_+^Y$ is the integral Goldie constant \eqref{eq:C-plus-Y} of the stationary measure $\cL(Y_\infty)$.
\end{theorem}

Let us explain the logic of the proof. The reduction (Lemma \ref{lem:Y-structure}) reduces the ruin probability to the tail of the supremum of the chain $(Y_n)$; the Kesten--Goldie theorem yields the power-law tail of the stationary measure $\cL(Y_\infty)$; Goldie's result for the maximum equation transfers the power-law tail to the supremum $M^*$. Finally, the two-sided bound and a limiting argument give the exact equality for $\Psi(u)=\bbp(M^*\ge u)$. Technically, the proof consists of the following lemmas:
\begin{itemize}[leftmargin=*,itemsep=2pt]
\item Lemma \ref{lem:moments-detailed}: moment identities for the pair $(M_1,Q_1)$;
\item Lemma \ref{lem:Y-inf-asymptotic}: the asymptotics of the stationary measure $\bbp(Y_\infty>u)\sim C_+^Y u^{-\beta}$ by the Kesten--Goldie theorem \cite[Theorem 2.4.4]{BDM2016};
\item Lemma \ref{lem:sup-asymptotic}: the asymptotics of the supremum $\bbp(M^*>u)\sim C^* u^{-\beta}$ by \cite[Theorem 6.2]{Goldie1991};
\item the proof of Theorem \ref{thm:main} itself, with the passage from $\bbp(M^*>u)$ to $\bbp(M^*\ge u)$ via the two-sided bound and a limiting argument.
\end{itemize}

As a meaningful comment on the exponent $\beta$, we note the following observation, which gives it a transparent L\'evy interpretation.

\begin{remark}[L\'evy interpretation of the exponent $\beta$]
Condition (A2), $H(\beta)=0$, has an equivalent form in terms of the Laplace exponent $\psi_V$ of the L\'evy process $V$ itself. Since $V$ is a L\'evy process and $T_1$ is independent of $V$,
\[
\bbe[e^{-\beta V_{T_1}}]=\bbe\bigl[\bbe[e^{-\beta V_t}]\big|_{t=T_1}\bigr]=\bbe[e^{T_1\psi_V(\beta)}].
\]
The identity $\bbe[e^{T_1\psi_V(\beta)}]=1$, together with $T_1>0$ a.s. and the strict monotonicity of the function $s\mapsto\bbe[e^{sT_1}]$, implies $\psi_V(\beta)=0$. In substance: the exponent $\beta$ is the Cram\'er root of the cumulant $\psi_V$ of the process $V$ itself, rather than of the random walk $V_{T_n}$. The distribution of the inter-jump times affects only the value of the constant $C^*$ in \eqref{eq:main}, not the exponent.
\end{remark}

\begin{lemma}\label{lem:moments-detailed}
Under assumptions \textup{(A1)--(A3)} the following hold:
\begin{itemize}[leftmargin=*,itemsep=2pt]
\item[\textup{(i)}] $\bbe[M_1^\beta]=1$;
\item[\textup{(ii)}] $\bbe[M_1^\beta\,|\log M_1|]<\infty$;
\item[\textup{(iii)}] $m:=\bbe[M_1^\beta\log M_1]=H'(\beta)>0$;
\item[\textup{(iv)}] $\bbe[|Q_1|^\beta]<\infty$.
\end{itemize}
\end{lemma}

\begin{proof}
(i) is a reformulation of the definition of $\beta$:
\[
\bbe[M_1^\beta]=\bbe[e^{-\beta V_{T_1}}]=e^{H(\beta)}=1.
\]
(ii) follows from $\beta\in\mathrm{int}(\mathrm{dom}\,H)$ (assumption (A2)): there exists $\delta>0$ such that $\beta\pm\delta\in\mathrm{dom}\,H$. Then the elementary inequality
\begin{equation}\label{eq:tail-split}
|x|\,e^{-\beta x}\le\frac1\delta\bigl(e^{-(\beta-\delta)x}+e^{-(\beta+\delta)x}\bigr),\qquad x\in\bbr,
\end{equation}
holds, being equivalent (after division by $e^{-\beta x}$) to $|\delta x|/2\le\cosh(\delta x)$. Applying \eqref{eq:tail-split} to $V_{T_1}$ and taking expectations,
\begin{multline*}
\bbe[|V_{T_1}|\,e^{-\beta V_{T_1}}]\le\frac1\delta\bigl(\bbe[e^{-(\beta-\delta)V_{T_1}}]+\bbe[e^{-(\beta+\delta)V_{T_1}}]\bigr)=
\\
=\frac1\delta\bigl(e^{H(\beta-\delta)}+e^{H(\beta+\delta)}\bigr)<\infty.
\end{multline*}
Since $|\log M_1|=|V_{T_1}|$ and $M_1^\beta=e^{-\beta V_{T_1}}$, this gives $\bbe[M_1^\beta\,|\log M_1|]<\infty$.

(iii) is a direct computation (differentiation under the integral sign is justified by (ii)):
\[
m=\bbe[M_1^\beta\log M_1]=\bbe[e^{-\beta V_{T_1}}\cdot(-V_{T_1})]=-\bbe[V_{T_1}\,e^{-\beta V_{T_1}}]=H'(\beta).
\]
For the positivity $H'(\beta)>0$: $H$ is strictly convex on the interior of its domain (strict convexity follows from the non-degeneracy of $V_{T_1}$, ensured by (A1)), and $H(0)=H(\beta)=0$; by Rolle's theorem there exists $\xi\in(0,\beta)$ with $H'(\xi)=0$, and from the strict increase of $H'$ we obtain $H'(\beta)>H'(\xi)=0$.

(iv) is the content of Lemma 2.1 and Corollary 2.2 of \cite{EKS2022}.
\end{proof}

\begin{lemma}\label{lem:Y-inf-asymptotic}
Under assumptions \textup{(A1)--(A3)} and \textup{(N)} there exists $C_+^Y\in(0,\infty)$ such that
\[
\bbp(Y_\infty>u)\sim C_+^Y\,u^{-\beta}\qquad\text{as }u\to\infty,
\]
and
\begin{equation}\label{eq:C-plus-Y}
C_+^Y=\frac1{\beta\,H'(\beta)}\,\bbe\bigl[((M_1Y_\infty+Q_1)_+)^\beta-((M_1Y_\infty)_+)^\beta\bigr],
\end{equation}
where $Y_\infty$ on the right-hand side of \eqref{eq:C-plus-Y} is independent of $(M_1,Q_1)$.
\end{lemma}

\begin{proof}
We apply the Kesten--Goldie theorem in its one-dimensional form \cite[Theorem 2.4.4]{BDM2016} (going back to Theorem 5 of \cite{Kesten1973} and Theorem 4.1 of \cite{Goldie1991}) to the stationary measure $\cL(Y_\infty)$ of equation \eqref{eq:fixedpoint}. The conditions of \cite[Theorem 2.4.4]{BDM2016} are:
\begin{itemize}[leftmargin=*,itemsep=2pt]
\item $M_1>0$ almost surely ($M_1=e^{-V_{T_1}}$);
\item $\log M_1=-V_{T_1}$ is non-arithmetic---condition (N);
\item $\bbe[M_1^\beta]=1$, $\bbe[|Q_1|^\beta]<\infty$, $\bbe[M_1^\beta\log_+M_1]<\infty$---Lemma \ref{lem:moments-detailed};
\item $\bbp(M_1x+Q_1=x)<1$ for every $x\in\bbr$. This is verified directly: conditionally on $V_{T_1}$ (hence on $M_1$), the variable $Q_1$ remains random because of the term $-e^{-V_{T_1-}}\xi_1$, into which the non-degenerate random variable $\xi_1$ enters ($F_\xi(\{0\})=0$); therefore $\bbp(Q_1=x(1-M_1))<1$ for every $x$.
\end{itemize}
By \cite[Theorem 2.4.4]{BDM2016}, $\bbp(Y_\infty>u)\sim C_+^Y u^{-\beta}$ with the explicit formula \eqref{eq:C-plus-Y}. The positivity $C_+^Y>0$ is ensured by the unboundedness of the support of $\cL(Y_\infty)$ from above. By \cite[Theorem 2.5.5, Theorem 2.4.6]{BDM2016}, the support of the stationary measure is a half-line or the whole real line, and if it is unbounded from above, then $C_+^Y>0$. The unboundedness from above of the support for the model under consideration is established in \cite[Theorem 3.1]{KLP2026}, relying on Proposition 2.5.4 of \cite{BDM2016}. Formally, \cite{KLP2026} considers the mixed model, but the non-life case is a particular case of it: the conditions $c>0$ and $\xi_i<0$ imply $\bbp(Q_1>0)>0$, which ensures the applicability of Proposition 2.5.4.
\end{proof}

\begin{lemma}\label{lem:sup-asymptotic}
Under assumptions \textup{(A1)--(A3)} and \textup{(N)} there exists $C^*\in[C_+^Y,\infty)$ such that
\[
\bbp(M^*>u)\sim C^*\,u^{-\beta}\qquad\text{as }u\to\infty.
\]
\end{lemma}

\begin{proof}
We apply Goldie's result \cite[Theorem 6.2]{Goldie1991} on the solution of the distributional equation $M\stackrel{d}{=}(M_1M+Q_1)_+$. This is a particular case of Letac's model from \cite{Goldie1991} with $L=-Q_1/M_1$. The pair $(M_1,Q_1)$ satisfies the same Kesten--Goldie conditions as in Lemma \ref{lem:Y-inf-asymptotic} and Lemma \ref{lem:moments-detailed}. The additional moment condition $\bbe[(M_1L^+)^\beta]=\bbe[(Q_1^-)^\beta]\le\bbe[|Q_1|^\beta]<\infty$ holds automatically. Under these assumptions there exists a unique solution $M$ of this equation, and
\[
\bbp(M>u)\sim C^*u^{-\beta}\qquad\text{as }u\to\infty.
\]
The identification $M\stackrel{d}{=}M^*=\sup_n Y_n$ is described in \cite[discussion after Eq. (5.6.32)]{BDM2016}. The iterated sequence $Z_n:=(M_nZ_{n-1}+Q_n)_+$ with $Z_0:=0$ satisfies $Z_n\stackrel{d}{=}\max_{0\le j\le n}Y_j$, since the pairs $(M_i,Q_i)$ are independent and identically distributed. The monotone maximum $\max_{0\le j\le n}Y_j$ increases to $\sup_n Y_n=M^*$ almost surely, so $Z_n$ converges in distribution to $M^*$. By Lemma \ref{lem:Y-convergence} we have $Y_\infty^{\mathrm{fwd}}=\lim Y_n\le\sup_n Y_n=M^*$ almost surely, whence $\bbp(M^*>u)\ge\bbp(Y_\infty>u)$ for all $u$, and therefore $C^*\ge C_+^Y$.
\end{proof}

\begin{proof}[Proof of Theorem \ref{thm:main}]
By Lemma \ref{lem:Y-structure}, in the non-life insurance case $\Psi(u)=\bbp(M^*\ge u)$. By Lemma \ref{lem:sup-asymptotic}, $\lim_{u\to\infty}u^\beta\bbp(M^*>u)=C^*$. We establish 
\[
\lim_{u\to\infty}u^\beta\bbp(M^*\ge u)=C^*
\]
via a two-sided bound and a limiting argument. For any $\eps>0$,
\[
\bbp(M^*>u)\le\bbp(M^*\ge u)\le\bbp(M^*>u-\eps).
\]
Multiplying by $u^\beta$ and passing to the limit as $u\to\infty$,
\begin{multline*}
C^* = \lim u^\beta\bbp(M^*>u) \le \liminf u^\beta\bbp(M^*\ge u)\le \limsup u^\beta\bbp(M^*\ge u) \le
\\
\le \lim u^\beta\bbp(M^*>u-\eps)=C^*,
\end{multline*}
where the last equality uses $u^\beta\sim(u-\eps)^\beta$ as $u\to\infty$. Hence $u^\beta\bbp(M^*\ge u)\to C^*$, which proves \eqref{eq:main}.
\end{proof}

\section{Connection with exact analytic solutions in the Cram\'er--Lundberg model with investments}\label{sec:CLgBm}

We compare the established result with the exact analytic solution for the classical submodel---the Cram\'er--Lundberg model with proportional investment in a geometric Brownian motion, studied in \cite{AK2026,PRYu2026}. In this submodel the inter-jump times are exponentially distributed with intensity $\lambda>0$, the price of the risky asset is a geometric Brownian motion, and a fixed fraction $\kappa\in(0,1]$ of the capital is invested in the risky asset, the rest in a riskless account with rate $\tilde r\ge0$. The effective log-value of the portfolio is
\[
V_t=\bigl(\tilde a_\kappa-\tfrac12\tilde\sigma_\kappa^2\bigr)t+\tilde\sigma_\kappa W_t,\qquad\tilde a_\kappa:=\tilde r+\kappa(\tilde a-\tilde r),\;\;\tilde\sigma_\kappa:=\kappa\tilde\sigma,
\]
a purely diffusive L\'evy process; the jumps $\xi_i$ are strictly negative when $c>0$. The Laplace exponent of $V$ is
\[
\psi_V(q)=-\bigl(\tilde a_\kappa-\tfrac12\tilde\sigma_\kappa^2\bigr)q+\tfrac{\tilde\sigma_\kappa^2}{2}q^2.
\]
The equation $\psi_V(\beta)=0$ with $\beta>0$ gives
\begin{equation}\label{eq:beta-CLgBm}
\beta=\frac{2\tilde a_\kappa}{\tilde\sigma_\kappa^2}-1=\frac{2(\tilde r+\kappa(\tilde a-\tilde r))}{(\kappa\tilde\sigma)^2}-1=\gamma-1,
\end{equation}
where $\gamma:=2\tilde a_\kappa/\tilde\sigma_\kappa^2$ is the standard notation of \cite{AK2026,PRYu2026}. By Theorem \ref{thm:main}, $\Psi(u)\sim C^*u^{-\beta}=C^*u^{-\gamma+1}$, which agrees with the known result of \cite{AK2026}; there this asymptotics is established by analytic methods through reduction to a Volterra integral equation, and the constant is obtained explicitly via the solution of a special integral equation. In the more general Sparre Andersen model with an arbitrary distribution of the inter-jump times, the analytic approach via integro-differential equations loses its applicability, since there is no closed second-order equation for $\Psi$, whereas the method considered here, based on Goldie's theorem and the asymptotics of the supremum of the perpetuity, works without changes.

\paragraph{Model parameters and the exact representation.} For definiteness we fix the control values from \cite{PRYu2026}: $\lambda=1$, $\mu=1$ (the parameter of the exponential claim distribution), $\tilde a=0.15$, $\tilde\sigma=0.4$, $c=0.5$, $\tilde r=0.05$, $\kappa=0.9$. From \eqref{eq:beta-CLgBm},
\[
\beta=\frac{2\cdot0.14}{0.1296}-1\approx1.1605.
\]
By the exact-solution theorem of \cite{PRYu2026}, the survival probability $\Phi(u)=1-\Psi(u)$ for this model is expressed through an integral of the regular branch of the double confluent Heun function $\mathrm{HeunD}$ \cite{slavyanov} with parameters determined by the set $(\tilde a,\tilde\sigma,c,\kappa,\tilde r,\lambda,\mu)$.

\paragraph{Agreement of the asymptotics.} From the properties of the analytic continuation of the Heun function near the irregular singular point $\zeta=\infty$, studied in \cite[Proposition 1]{PRYu2026}, it follows that the ruin probability has the power-law asymptotics
\[
\Psi(u)\sim C_\infty\,u^{-\gamma+1}\quad\text{as }u\to\infty,
\]
where $\gamma=2((\tilde a-\tilde r)\kappa+\tilde r)/(\kappa\tilde\sigma)^2$. In our notation $\gamma-1$ coincides exactly with the exponent $\beta$, and the constant $C_\infty$, defined in \cite{PRYu2026} as the normalizing factor in the representation via Heun functions, coincides with $C^*$. Thus, Theorem \ref{thm:main} agrees with the asymptotics obtained in \cite{PRYu2026} by the analytic theory of differential equations. The coincidence of the two methodologies---Goldie's implicit renewal theory and the analysis of special functions---on their common intersection serves as a nontrivial cross-check of both methods. We note that the work does not recompute the constant itself by the second method: the exponent and the internal consistency of the tail are confirmed numerically, while the equality $C_\infty=C^*$ is established analytically.

\paragraph{Numerical illustration of the agreement.} For the parameters above, the survival probability $\Phi(u)=1-\Psi(u)$ is found by numerically integrating the corresponding ODE in canonical form
\[
u^2g''+[\mu u^2+(\gamma+2)u+\theta]g'+[\mu\gamma u+(\mu\theta+\gamma-\nu)]g=0
\]
from \cite{PRYu2026} (where, for convenience, we have renamed the ODE coefficients $\theta:=2c/(\kappa\tilde\sigma)^2$ and $\nu:=2\lambda/(\kappa\tilde\sigma)^2$, to avoid confusion with the Cram\'er exponent $\beta=\gamma-1$); this is neither Monte Carlo simulation nor an approximation, but a direct solution of the Cauchy problem, equivalent to evaluating the Heun function. The limiting constant $C^*$ is extracted by extrapolating the tail coefficient $u^\beta\Psi(u)$ as $u\to\infty$ (Richardson extrapolation), which gives $C^*\approx10.674$. Table \ref{tab:pryu-numerics} lists the values of $\Psi(u)$ and $u^\beta\Psi(u)$; one observes a stable convergence $u^\beta\Psi(u)\to C^*$, with the remainder $u^\beta\Psi(u)-C^*$ decaying roughly like $u^{-1}$ (as $u$ doubles, the relative error decreases roughly by half)---in qualitative agreement with the known rate-of-convergence estimates for perpetuities \cite[Theorem 3.2]{Goldie1991}. A rigorous transfer of such estimates to $\Psi(u)$ is a separate question, see Section \ref{sec:discussion}. Figure \ref{fig:pryu-loglog} shows the graph of $\Psi(u)$ on a double logarithmic scale: for large $u$ the curve approaches a straight line of slope $-\beta\approx-1.1605$ and coincides with the asymptote $C^*u^{-\beta}$, which constitutes the geometric content of Theorem \ref{thm:main}.

\begin{table}[ht]
\centering
\caption{Numerical verification of the asymptotics $\Psi(u)\sim C^*u^{-\beta}$ for the Cram\'er--Lundberg model with proportional investment ($\kappa=0.9$, $\beta\approx1.1605$). Parameters: $\mu=\lambda=1$, $\tilde a=0.15$, $\tilde\sigma=0.4$, $c=0.5$, $\tilde r=0.05$. Limiting constant $C^*\approx10.674$.}
\label{tab:pryu-numerics}
\begin{tabular}{cccc}
\toprule
$u$ & $\Psi(u)$ & $u^{\beta}\Psi(u)$ & $(u^{\beta}\Psi(u)-C^*)/C^*$\\
\midrule
$10$    & $4.77\times10^{-1}$ & $6.903$  & $-3.53\times10^{-1}$\\
$50$    & $1.05\times10^{-1}$ & $9.810$  & $-8.10\times10^{-2}$\\
$100$   & $4.89\times10^{-2}$ & $10.237$ & $-4.09\times10^{-2}$\\
$200$   & $2.23\times10^{-2}$ & $10.454$ & $-2.06\times10^{-2}$\\
$500$   & $7.81\times10^{-3}$ & $10.586$ & $-8.24\times10^{-3}$\\
$1000$  & $3.51\times10^{-3}$ & $10.630$ & $-4.10\times10^{-3}$\\
$2000$  & $1.57\times10^{-3}$ & $10.652$ & $-2.02\times10^{-3}$\\
\bottomrule
\end{tabular}
\end{table}

\begin{figure}[ht]
\centering
\includegraphics[width=0.78\textwidth]{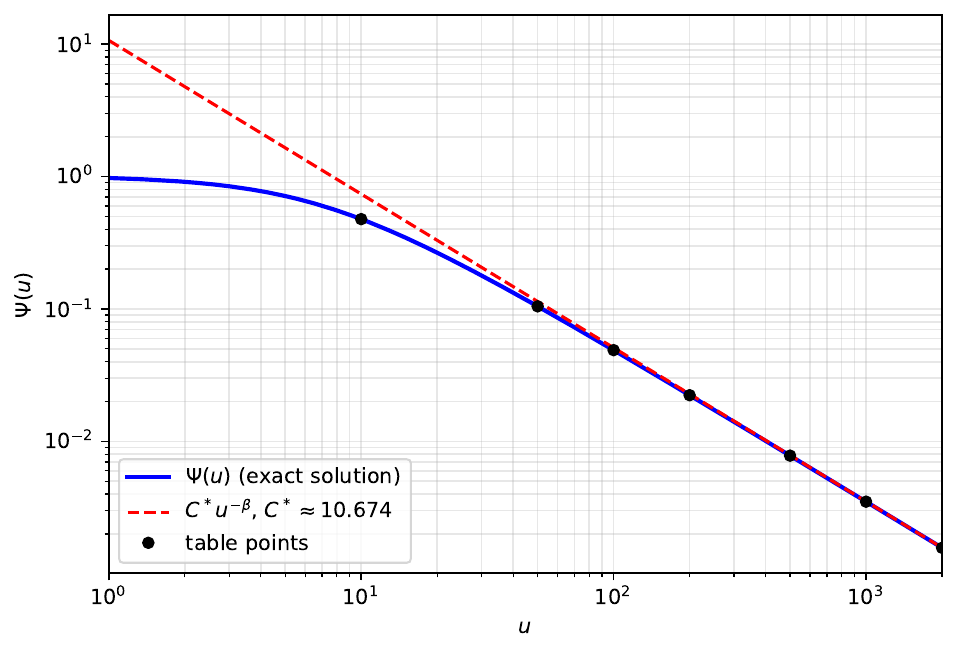}
\caption{The ruin probability $\Psi(u)$ and its asymptote $C^*u^{-\beta}$ on a double logarithmic scale ($\kappa=0.9$).}
\label{fig:pryu-loglog}
\end{figure}

\section{Discussion}\label{sec:discussion}

The established result---the exact equality $\lim u^\beta\Psi(u)=C^*$---gives a complete description of the leading term of the asymptotics of the ruin probability in the Sparre Andersen model for non-life insurance with investments in an arbitrary L\'evy process. The exponent $\beta$ is the Cram\'er root of the Laplace exponent $\psi_V$ of the L\'evy process $V$, independent of the distribution of the inter-jump times. The distribution $F$ affects only the value of the constant, not the exponent. At the same time, no explicit formula for $C^*$ itself is available in the general model (see below); only the lower bound $C^*\ge C_+^Y$ is known, where $C_+^Y$ has the integral representation \eqref{eq:C-plus-Y}. Compared with the two-sided estimate \eqref{eq:liminflimsup} of \cite{EKS2022,KLP2026}, the new content is precisely the existence of an exact limit rather than of the order of the asymptotics.

The natural directions for continuing the research are the following.

\textit{Annuity models.} For annuity payments ($c<0$, $\xi_i>0$ a.s.) the trajectories of $Y$ have a positive continuous drift between jumps and negative jumps at the moments $T_n$, so that $\sup_{t\ge0}Y_t\ne\sup_{n\ge0}Y_{T_n}$ in general. The discrete reduction of Lemma \ref{lem:Y-structure} ceases to work directly: ruin may occur continuously between jumps, and for the Markov property one has to consider the pair $(X_t,D_t)$ with the age process $D_t$. An analytic study of the annuity model via an integro-differential equation was carried out in \cite{P2601}; the transfer of the result of the present work (by the probabilistic method) to the annuity case requires a substantial adaptation of the technique and is the content of a separate paper.

\textit{The mixed model.} When $F_\xi$ charges both half-axes, $Y_\infty$ takes values of both signs, and the question of the two-sided tail of $\cL(Y_\infty)$ arises: the left tail $\bbp(Y_\infty<-u)$ has the same power-law form $\sim C^- u^{-\beta}$ under appropriate conditions on the support of $\cL(Y_\infty)$ (a two-sided version of Proposition 2.5.4 of \cite{BDM2016}). This topic is the content of a separate paper.

\textit{Quantitative rate-of-convergence estimates.} The limiting equality obtained does not provide a quantitative estimate of the rate of convergence $|u^\beta\Psi(u)-C^*|\le Ku^{-\delta}$. Such estimates in perpetuity asymptotics are known \cite[Theorem 3.2]{Goldie1991} under additional conditions on the moment characteristics of the pair $(M_1,Q_1)$. Their transfer to $\Psi(u)$ through the links we have built requires separate work.

\textit{An explicit formula for $C^*$.} Obtaining an explicit (computable) formula for the constant $C^*$ in the general model is an open question. In the Cram\'er--Lundberg model with geometric Brownian motion an exact formula is established by other methods (see Section \ref{sec:CLgBm}); the agreement of the two approaches on their common intersection gives a nontrivial cross-check of both programs.

Another meaningful direction is the relaxation of the moment condition on the claim sizes $\bbe[|\xi|^\beta]<\infty$, which excludes heavy-tailed claims with regularly varying tails. This direction already has some results. In \cite{HultLindskog2011} a model with investments in general semimartingales and heavy-tailed claims is considered. In \cite{GaierGrandits2004} concrete asymptotic estimates are obtained for a model with investments at a constant interest rate under regularly varying claim tails. The theory of regularly varying functions and their probabilistic applications is presented in \cite{BinghamGoldieTeugels1989}, and the classical theory of extremal events in insurance and finance in \cite{EKM1997}. The transfer of the apparatus of the present work to the heavy-tailed case is the content of a separate study. When the condition $\bbe[|\xi|^\beta]<\infty$ is dropped, Goldie's theorem ceases to give a purely power-law asymptotics with exponent $\beta$, and the asymptotic behavior is determined by the competition between the claim tail and the Cram\'er tail of the perpetuity. When the tail of $Q_1$ varies regularly with exponent less than $\beta$, the tail of $Y_\infty$ inherits the same regular variation. This result was first established by Grincevicius \cite{Grincevicius1975}. An improved proof and the converse direction were given by Grey \cite{Grey1994}.

\section*{Competing interests}
The author declares no competing interests.


\begin{thebibliography}{99}

\bibitem{AK2026}
Antipov, V., Kabanov, Yu.: On the integro-differential equation arising in the ruin problem for non-life insurance models with investment. Mathematics, {\bf 14}, 1035 (2026)

\bibitem{AsmussenAlbrecher2010}
Asmussen, S., Albrecher, H.: Ruin Probabilities, 2nd ed. World Scientific, Singapore (2010)

\bibitem{BinghamGoldieTeugels1989}
Bingham, N.H., Goldie, C.M., Teugels, J.L.: Regular Variation. Cambridge University Press, Cambridge (1989)

\bibitem{BDM2016}
Buraczewski, D., Damek, E., Mikosch, Th.: Stochastic Models with Power-Law Tails. The Equation $X=AX+B$. Springer Series in Operations Research and Financial Engineering. Springer, Berlin (2016)

\bibitem{EKS2022}
Eberlein, E., Kabanov, Yu., Schmidt, T.: Ruin probabilities for a Sparre Andersen model with investments. Stochastic Processes and their Applications, {\bf 144}, 72--84 (2022)

\bibitem{EKM1997}
Embrechts, P., Kl\"uppelberg, C., Mikosch, T.: Modelling Extremal Events for Insurance and Finance. Springer, Berlin (1997)

\bibitem{Feller1971}
Feller, W.: An Introduction to Probability Theory and Its Applications, Vol. II, 2nd ed. Wiley, New York (1971)

\bibitem{GaierGrandits2004}
Gaier, J., Grandits, P.: Ruin probabilities and investment under interest force in the presence of regularly varying tails. Scandinavian Actuarial Journal, 256--278 (2004)

\bibitem{Goldie1991}
Goldie, C.M.: Implicit renewal theory and tails of solutions of random equations. Annals of Applied Probability, {\bf 1}, 126--166 (1991)

\bibitem{Grandits2004}
Grandits, P.: A Karamata-type theorem and ruin probabilities for an insurer investing proportionally in the stock market. Insurance: Mathematics and Economics, {\bf 34}, 297--305 (2004)

\bibitem{Grey1994}
Grey, D.R.: Regular variation in the tail behaviour of solutions of random difference equations. Annals of Applied Probability, {\bf 4}, 169--183 (1994)

\bibitem{Grincevicius1975}
Grincevicius, A.K.: One limit distribution for a random walk on the line. Lithuanian Mathematical Journal, {\bf 15}, 580--589 (1975)

\bibitem{HultLindskog2011}
Hult, H., Lindskog, F.: Ruin probabilities under general investments and heavy-tailed claims. Finance and Stochastics, {\bf 15}, 243--265 (2011)

\bibitem{KLP2026}
Kabanov, Yu., Legenkiy, D., Promyslov, P.: Distributional equations and the ruin problem for the Sparre Andersen model with investments. Extremes, {\bf 29}, 65--87 (2026)

\bibitem{KP2023}
Kabanov, Yu., Promyslov, P.: Ruin probabilities for a Sparre Andersen model with investments: the case of annuity payments. Finance and Stochastics, {\bf 27}, 887--902 (2023)

\bibitem{KalashnikovNorberg2002}
Kalashnikov, V., Norberg, R.: Power tailed ruin probabilities in the presence of risky investments. Stochastic Processes and their Applications, {\bf 98}, 211--228 (2002)

\bibitem{Kesten1973}
Kesten, H.: Random difference equations and renewal theory for products of random matrices. Acta Mathematica, {\bf 131}, 207--248 (1973)

\bibitem{Paulsen1993}
Paulsen, J.: Risk theory in a stochastic economic environment. Stochastic Processes and their Applications, {\bf 46}, 327--361 (1993)

\bibitem{Paulsen2002}
Paulsen, J.: On Cram\'er-like asymptotics for risk processes with stochastic return on investments. Annals of Applied Probability, {\bf 12}, 1247--1260 (2002)

\bibitem{PergamenshchikovZeitouni2006}
Pergamenshchikov, S., Zeitouni, O.: Ruin probability in the presence of risky investments. Stochastic Processes and their Applications, {\bf 116}, 267--278 (2006)

\bibitem{P2601}
Promyslov, P.: On the integro-differential equation arising in the ruin problem for annuity payment models. arXiv:2601.01447 (2026)

\bibitem{P2604}
Promyslov, P.: Existence of a classical solution to the integro-differential equation arising in the Cram\'er--Lundberg non-life insurance model with proportional investment. arXiv:2604.05143 (2026)

\bibitem{PRYu2026}
Promyslov, P., Romanov, M., Yurieva, G.: Exact solution of the ruin problem in the Cram\'er--Lundberg model with proportional investment. arXiv:2604.08745 (2026)

\bibitem{slavyanov}
Slavyanov, S.Yu., Lay, W.: Special Functions: A Unified Theory Based on Singularities. Oxford University Press, Oxford (2000)

\bibitem{Vervaat1979}
Vervaat, W.: On a stochastic difference equation and a representation of nonnegative infinitely divisible random variables. Advances in Applied Probability, {\bf 11}, 750--783 (1979)

\end{thebibliography}
\end{document}